\documentclass[12pt]{article}

\usepackage[english]{babel}
\usepackage[numbers,sort,compress]{natbib}
\usepackage[T1]{fontenc}
\usepackage{amsthm,amsmath,amssymb}
\usepackage[mathscr]{eucal}
\usepackage[utf8]{inputenc}
\usepackage{etoolbox}
\usepackage{latexsym}
\usepackage{graphicx}
\usepackage{geometry}
\usepackage{longtable}
\usepackage{array}     
\usepackage{pgffor}    
\usepackage{caption}
\usepackage{physics}
\newtheorem{theorem}{Theorem}[section]
\newtheorem{corollary}[theorem]{Corollary}
\newtheorem{proposition}[theorem]{Proposition}
\newtheorem{lemma}[theorem]{Lemma}
\newtheorem{definition}[theorem]{Definition}

\def \adj{\operatorname{adj}}
\def \tr{\operatorname{tr}}

\begin{document}
	\title{On enumeration of spanning trees of complete multipartite graphs containing a fixed spanning forest}
	
	\author{Wei Wang$^a$ \quad Jun Ge$^b$\thanks{Corresponding author. Email address: mathsgejun@163.com} \\[2mm]
		{\footnotesize  $^a$School of Mathematics, Physics and Finance, Anhui Polytechnic University, Wuhu 241000, China}\\		
		{\footnotesize  $^b$School of Mathematical Sciences, Sichuan Normal University, Chengdu 610066, China}
	}
	\date{}
	
	\maketitle
	
	\begin{abstract}
		 We present a determinantal formula for the number of spanning trees of a complete multipartite graph containing a given spanning forest $F$. Our approach relies on the Generalized Matrix Determinant Lemma and Jacobi's formula for the derivative of a determinant.  This work generalizes known results for complete bipartite graphs and offers an algebraic perspective on the problem.
	\end{abstract}
	
	
	
	\section{Introduction}
	All graphs considered in this paper are loopless, while parallel edges are allowed. For a graph $G$, we use $\mathcal{T}(G)$ to denote the set of spanning trees of $G$. Let $\tau(G)=|\mathcal{T}(G)|$, that is, the number of spanning trees of $G$. 
For a subgraph $H$ of $G$, we use $\mathcal{T}_H(G)$ to denote all spanning trees of $G$ that contain all edges in $H$. 
Accordingly, we write $\tau_H(G)=|\mathcal{T}_H(G)|$. By adding isolated vertices if necessary, we may safely assume that $H$ is a spanning subgraph. 
Note that if $H$ is the empty graph, then  $\mathcal{T}_H(G)$ coincides with $\mathcal{T}(G)$ and hence  $\tau_H(G)=\tau(G)$.
	
	Counting spanning trees in graphs is a classic problem in graph theory and has a close connection with many other fields in mathematics, 
statistical physics and theoretical computer sciences. See, for example, some recent work \cite{dong2022,li2023,li2025,gong2018,dong2017}.
	
	The celebrated Cayley's formula \cite{cayley1889} states that $\tau(K_n)=n^{n-2}$. 
For a complete bipartite graph, Fiedler and  Sedl\'a\v{c}ek \cite{fiedler1958} showed that \begin{equation}\label{b}
	\tau(K_{n_1,n_2})=n_1^{n_2-1}n_2^{n_1-1}.
	\end{equation} This result was further extended,  by various authors using different means \cite{austin1960,lewis1999,biggs1993,klee2019} to complete multipartite  graphs as follows:
	\begin{equation*}
	\tau(K_{n_1,n_2,\ldots,n_s})=n^{s-2}\prod_{i=1}^s(n-n_i)^{n_i-1},
	\end{equation*}
	where $s\ge 2$ and $n=n_1+\cdots+n_s$.
	
	Let $F$ be a spanning forest of a complete graph $K_n$ whose components are $T_1,\ldots, T_c$. Moon \cite{moon1964, moon1967} proved that 
	\begin{equation*}
		\tau_F(K_n)=n^{c-2}\prod_{p=1}^{c}m_p,
	\end{equation*}
	where $m_p=|V(T_p)|$ is the order of $T_p$. We note that if $F$ is empty, then Moon's formula reduces to Cayley's formula. 
	
	The Moon-type formula for complete bipartite graphs was found by Dong and Ge \cite{dong2022}. Explicitly, it states that, for a given spanning forest $F$ of $K_{n_1,n_2}$ whose components are $T_1,\ldots,T_c$,
	\begin{equation}\label{m2}
		\tau_F(K_{n_1,n_2})=\frac{1}{n_1n_2}\left(\prod_{p=1}^c\left(n_{1p}n_2+n_{2p}n_1\right)\right)\left(1-\sum_{p=1}^c \frac{n_{1p}n_{2p}}{n_{1p}n_2+n_{2p}n_1}\right),
	\end{equation} 
	where $n_{ip}=|X_i\cap V(T_p)|$ and $(X_1,X_2)$ is the bipartition of $K_{n_1,n_2}$ with $|X_i|=n_i$ for $i=1,2$. For the trivial case that $F$ is empty, we see that $(n_{1p},n_{2p})$ equals $(1,0)$ or $(0,1)$ and hence   Eq.~\eqref{m2} reduces to Eq.~\eqref{b}.
	
	The original proof of Eq.~\eqref{m2} given by Dong-Ge \cite{dong2022} is rather technical,  which is based on the Inclusion-Exclusion Principle and various complex algebraic identities. Two other simple proofs of Eq.~\eqref{m2} were obtained by Li-Chen-Yan \cite{li2023} and Li-Yan \cite{li2025} using the mesh-star transformation and a variant of the Teufl-Wagner formula, respectively. In particular, Li, Chen and Yan \cite{li2023} obtained a Moon-type formula for complete $s$-partite graphs. They showed that  $\tau_F(K_{n_1,\ldots,n_s})$ (where $F$ is a spanning forest) can be expressed by the sum of weights of spanning trees of a particular edge-weighted complete graph $K_s$. A much shorter proof of Eq.~\eqref{m2} was found by Ge \cite{ge2025} recently, using a form of Matrix Tree Theorem due to Klee and Stamps \cite{klee2019} based on the Matrix Determinant Lemma.

	The current paper is a natural extension of the algebraic techniques developed by Klee-Stamps \cite{klee2019} and Ge \cite{ge2025}.  A key finding  is that the number of spanning trees of a complete $s$-partite graph containing a fixed spanning forest is closely related to the determinant of a diagonal matrix by a rank-$s$ update.
	We utilize the Generalized Matrix Determinant Lemma to analyze the characteristic polynomial of the Laplacian matrix. This allows us to derive a compact determinantal formula, which can be seen as a variant of the Li-Chen-Yan formula  in \cite{li2023}.
	
	\section{Preliminaries}
	
For a graph $G$ on vertex set $\{v_1,\ldots,v_n\}$ (parallel edges allowed), let $L(G)$ be its Laplacian matrix, i.e., the diagonal entry $l_{ii}$ is the number of edges incident to $v_i$ and the off-diagonal entry $l_{ij}$ is the opposite of the number of the edges between $v_i$ and $v_j$. We use $\Phi(G;x)=\det(xI_n-L(G))$ to denote the Laplacian characteristic polynomial of $G$. For a matrix $M$, we use $\adj M$ to denote the adjugate matrix of $M$, i.e., the $(i,j)$ entry of $\adj M$ is the cofactor  corresponding to the $(j,i)$ entry of $M$. The all-ones vector and all-ones matrix will be denoted by $\mathbf{e}$ and $J$, respectively.
	
	\begin{lemma}[Matrix-Tree Theorem \cite{kirchhoff1847,biggs1993}] \label{mtt}
		Every cofactor of $L(G)$ is equal to the number of spanning trees of $G$, that is, $\adj L(G)=\tau(G)J$.
		\end{lemma}
	Let $a_1=\frac{\mathrm{d}}{\mathrm{d}x}\Phi(G;x)|_{x=0}$, that is, $a_1$ is the linear coefficient of $\Phi(G;x)$. Noting that $(-1)^{n-1}a_1$ is equal to the sum of the $n$ diagonal entries of $\adj L(G)$, the following corollary follows immediately from Lemma \ref{mtt}.
	\begin{corollary}\label{df}
	\begin{equation}	\tau(G) = \frac{(-1)^{n-1}}{n} \cdot \left. \frac{\mathrm{d}}{\mathrm{d}x}\Phi(G;x) \right|_{x=0}.
			\end{equation}
	\end{corollary}

	Let $G$ be a graph and $F$ be a spanning forest of $G$ whose components are $T_1,\ldots, T_c$.  We use $G/F$ to denote the graph obtained from $G$ by contracting all edges in $F$, and removing all loops. In other words, the vertices of $G/F$ are $T_1,\ldots, T_c$, and two vertices $V_p$ and $V_q$ are connected by $|E(V(T_p),V(T_q) )|$ edges, where $E(V(T_p),V(T_q))$ collects all edges in $G$ that have one end in $V(T_p)$ and the other in $V(T_q)$. Clearly, there exists a natural one-to-one correspondence between $\mathcal{T}_F(G)$ and $\mathcal{T}(G/F)$. Thus, $\tau_F(G)=\tau(G/F)$ and  hence Corollary \ref{df} implies the following formula for $\tau_F(G)$.
	\begin{corollary}\label{taud}
		\begin{equation}
		\tau_F(G)= \frac{(-1)^{c-1}}{c} \cdot \left. \frac{\mathrm{d}}{\mathrm{d}x}\Phi(G/F;x) \right|_{x=0}.
		\end{equation}
	\end{corollary}
Our derivation of the explicit formula for $\tau_F(G)$ relies on analyzing the structure of the characteristic polynomial via matrix calculus. The following two standard results from matrix analysis---the Generalized Matrix Determinant Lemma and Jacobi's formula---are the key algebraic tools we will employ to compute the characteristic polynomial and its derivative.
\begin{lemma}[Generalized Matrix Determinant Lemma {\cite[Sec.~18.1]{harville2008}}] \label{gmdl}
	Let $A$ be an invertible $n \times n$ matrix, and let $U, V$ be $n \times s$ and $s\times n$ matrices, respectively. Then
	\begin{equation*}
		\det(A + UV) = \det(A) \det(I_s + V A^{-1} U).
	\end{equation*}
\end{lemma}
	
	\begin{lemma}[Jacobi's formula {\cite[Sec.~15.8]{harville2008}}] \label{jac}
		Let $A(x)$ be a differentiable matrix function. Then
		\begin{equation*}
			\frac{\mathrm{d}}{\mathrm{d}x} \det A(x)  = \tr\left(\adj(A(x)) \frac{\mathrm{d}A(x)}{\mathrm{d}x}\right).
		\end{equation*}
	\end{lemma}

	\section{The structure of $L(K_{n_1,\ldots,n_s}/F)$}
		Let  $X_1, \dots, X_s$ be the partition sets of $ K_{n_1, \dots, n_s}$ ($s\ge 2$)  with $|X_i|=n_i\geq 1$ for $i=1,\ldots,s$. Suppose $F$ is a spanning forest of $K_{n_1, \dots, n_s}$ with $c$ components $T_1,\ldots,T_c$.  We define an $s\times c$ matrix $N$, whose $(i,p)$ entry is 
		\begin{equation*}
			n_{ip}=|X_i\cap V(T_p)|,
		\end{equation*}
		that is, $n_{ip}$ is the number of the vertices of $X_i$ lying in the component $T_p$. For convenience, set
		\begin{equation*}
			n=\sum_{i=1}^s n_i, \quad m_p=|V(T_p)|, \quad\text{and}\quad \alpha_p=nm_p-\sum_{i=1}^s n_i n_{ip} \text{~for~} p=1,\ldots,c.
		\end{equation*}
		The following identities are clear from the definitions.
				\begin{equation}\label{rfp}
			\sum_{p=1}^{c}m_p=\sum_{i=1}^s n_i=n
		,\quad
			\sum_{p=1}^c n_{ip}=n_i, \quad \text{and}\quad \sum_{i=1}^{s}n_{ip}=m_p.
			\end{equation}
			
We present two equivalent descriptions of $\alpha_p$.
			\begin{lemma}\label{pos}
				$\alpha_p=\sum_{i=1}^s n_i(m_p-n_{ip})=\sum_{i=1}^s(n-n_i)n_{ip}$. In particular, $\alpha_p>0$  for each $p\in\{1,\ldots,c\}$.
			\end{lemma}
			\begin{proof}
				As $n=\sum_{i=1}^s n_i$, we obtain 
				\begin{equation*}
					\alpha_p=nm_p-\sum_{i=1}^s n_i n_{ip} =\sum_{i=1}^s n_im_p-\sum_{i=1}^s n_i n_{ip} =\sum_{i=1}^sn_i(m_p-n_{ip}).
					\end{equation*}
					Similarly, from the equality $m_p=\sum_{i=1}^s n_{ip}$, we see that  \begin{equation*}
				\alpha_p=n \sum_{i=1}^s n_{ip}-\sum_{i=1}^s n_i n_{ip} =\sum_{i=1}^s(n-n_i) n_{ip}\ge \sum_{i=1}^s n_{ip}=m_p>0.
				\end{equation*}
				This completes the proof.
			\end{proof}
		A key fact is that the Laplacian matrix of the graph $K_{n_1,\ldots,n_s}/F$ can be written as a rank-$s$ perturbation of a diagonal matrix. 
We state this fact in the proposition below.
\begin{proposition}\label{Ls}
	Using the notation above,
	\begin{equation*}
		L(K_{n_1,\ldots,n_s}/F) = \operatorname{diag}(\alpha_1,\ldots, \alpha_c) + N^{\top}(I_s - J_s)N.
	\end{equation*}
\end{proposition}

\begin{proof}
	We prove this by comparing the corresponding entries of $L = L(K_{n_1,\ldots,n_s}/F)$ and $R = \operatorname{diag}(\alpha_1,\ldots, \alpha_c) + N^{\top}(I_s - J_s)N$, both of which are square matrices of order $c$. Let $l_{pq}$ and $r_{pq}$ denote the $(p,q)$ entries of $L$ and $R$, respectively. 
	
	For distinct $p,q \in \{1,\ldots,c\}$, let $w_{pq} = |E(V(T_p), V(T_q))|$ be the number of edges in $K_{n_1,\ldots,n_s}$ connecting a vertex in $T_p$ to a vertex in $T_q$. By the definition of the  graph $K_{n_1,\ldots,n_s}/F$, we have
	\begin{equation}\label{lp}
		l_{pq} = -w_{pq} \quad \text{for } p \neq q, \quad \text{and} \quad l_{pp} = \sum_{q \neq p} w_{pq}.
	\end{equation}
	
	Recall that in a complete multipartite graph, two vertices are adjacent if and only if they belong to different partitions. Thus, $w_{pq}$ is the total number of pairs $(u,v)$ with $u \in V(T_p), v \in V(T_q)$ minus those pairs belonging to the same partition:
	\begin{equation}\label{wp}
		w_{pq} = m_p m_q - \sum_{i=1}^{s} n_{ip} n_{iq} \quad \text{for } p \neq q.
	\end{equation}	
	On the other hand, a direct calculation shows that each off-diagonal entry $r_{pq}$ satisfies
	\begin{align}\label{rpq}
		r_{pq} &= (n_{1p},\ldots,n_{sp})(I_s - \mathbf{e}\mathbf{e}^{\top})(n_{1q},\ldots,n_{sq})^{\top} \nonumber \\
		&= \sum_{i=1}^{s} n_{ip} n_{iq} - \left( \sum_{i=1}^{s} n_{ip} \right) \left( \sum_{i=1}^{s} n_{iq} \right) \nonumber \\
		&= \sum_{i=1}^{s} n_{ip} n_{iq} - m_p m_q \nonumber \\
		&= -w_{pq}\nonumber\\
		&= l_{pq}.
	\end{align}
	This verifies that $L$ and $R$ have the same off-diagonal entries. It remains to show that $l_{pp} = r_{pp}$ for each $p$. From Eqs.~\eqref{lp} and \eqref{wp}, we have
	\begin{align}
		l_{pp} &= \sum_{q \neq p} w_{pq}\nonumber\\
		& = \sum_{q \neq p} \left( m_p m_q - \sum_{i=1}^{s} n_{ip} n_{iq} \right) \nonumber \\
		&= m_p \sum_{q \neq p} m_q - \sum_{i=1}^{s} n_{ip} \sum_{q \neq p} n_{iq} \nonumber \\
		&= m_p(n - m_p) - \sum_{i=1}^{s} n_{ip}(n_i - n_{ip}), \nonumber 
	\end{align}
	where the last equality follows from Eq.~\eqref{rfp}. Similarly to the derivation of Eq.~\eqref{rpq}, we have
	\begin{align}
		r_{pp} &= \alpha_p + (n_{1p},\ldots,n_{sp})(I_s - \mathbf{e}\mathbf{e}^{\top})(n_{1p},\ldots,n_{sp})^{\top} \nonumber \\
		&= \left( n m_p - \sum_{i=1}^s n_i n_{ip} \right) + \left( \sum_{i=1}^s n_{ip}^2 - \left( \sum_{i=1}^{s} n_{ip} \right)^2 \right) \nonumber \\
		&= n m_p - \sum_{i=1}^s n_i n_{ip} + \sum_{i=1}^s n_{ip}^2 - m_p^2 \nonumber \\
		&= m_p(n - m_p) - \sum_{i=1}^s n_{ip}(n_i - n_{ip}) \nonumber \\
		&= l_{pp}. \nonumber
	\end{align}
	This shows that $L$ and $R$ also share the same diagonal entries. Thus $L = R$, which completes the proof.
\end{proof}
\begin{proposition}\label{lcp}
	The Laplacian characteristic polynomial of $K_{n_1,\ldots,n_s}/F$ is given by
	\begin{equation*}
		\Phi(x) = \left( \prod_{p=1}^c (x - \alpha_p) \right) \det (I_s + Y(x)),
	\end{equation*}
	where $Y(x)$ is an $s \times s$ matrix whose $(i,j)$ entry is 
	\begin{equation}\label{yij}
		y_{ij} = \sum_{p=1}^c \frac{n_{jp}(m_p - n_{ip})}{x - \alpha_p}.
	\end{equation}
\end{proposition}

\begin{proof}
	By Proposition \ref{Ls} and the Generalized Matrix Determinant Lemma (Lemma \ref{gmdl}), we find that the Laplacian characteristic polynomial of $K_{n_1,\ldots,n_s}/F$ satisfies
	\begin{align}
		\Phi(x) &= \det\left( xI_c - \operatorname{diag}(\alpha_1,\ldots,\alpha_c) - N^{\top}(I_s - J_s)N \right) \nonumber \\
		&= \det\left( \operatorname{diag}(x - \alpha_1, \ldots, x - \alpha_c) + N^{\top}\cdot(J_s - I_s)N \right) \nonumber \\
		&= \left( \prod_{p=1}^c (x - \alpha_p) \right) \det\left( I_s + (J_s - I_s)N \operatorname{diag}\left( \frac{1}{x - \alpha_1}, \ldots, \frac{1}{x - \alpha_c} \right) N^{\top} \right). \nonumber
	\end{align}
	Let $B = N \operatorname{diag}\left( \frac{1}{x - \alpha_1}, \ldots, \frac{1}{x - \alpha_c} \right) N^{\top}$ and $Y = (J_s - I_s)B$. It suffices to show that the $(i,j)$ entry of $Y$ satisfies Eq.~\eqref{yij} for each $i,j \in \{1, \ldots, s\}$. Direct calculation shows that the $(k,j)$ entry of $B$ is
	\begin{equation*}
		b_{kj} = \sum_{p=1}^c \frac{n_{kp} n_{jp}}{x - \alpha_p}.
	\end{equation*}
	Since $Y = (J_s - I_s)B$, we have
	\begin{equation*}
		y_{ij} = \sum_{k \neq i} b_{kj} = \sum_{k \neq i} \sum_{p=1}^c \frac{n_{kp} n_{jp}}{x - \alpha_p} = \sum_{p=1}^c \frac{n_{jp}}{x - \alpha_p} \left( \sum_{k \neq i} n_{kp} \right).
	\end{equation*}
	Note that $\sum_{k \neq i} n_{kp} = m_p - n_{ip}$. Thus, Eq.~\eqref{yij} holds, and the proof of Proposition \ref{lcp} is complete.
\end{proof}
\begin{definition}\label{defZ}
	Let $Z(x)$ be the $s \times s$ matrix $I_s + Y(x)$ as defined in Proposition \ref{lcp}. That is, the $(i,j)$ entry of $Z(x)$ is 
	\begin{equation}\label{Z}
		z_{ij} = \delta_{ij} + \sum_{p=1}^c \frac{n_{jp}(m_p - n_{ip})}{x - \alpha_p},
	\end{equation} 
	where $\delta_{ij}$ is the Kronecker delta defined by
	\begin{equation*}
		\delta_{ij} = \begin{cases}
			1 & \text{if } i=j, \\
			0 & \text{if } i \neq j.
		\end{cases}
	\end{equation*}
\end{definition}

\begin{proposition}\label{Zsin}
	The matrix $Z(0)$ is singular, i.e., $\det Z(0) = 0$.
\end{proposition}

\begin{proof}
	By Lemma \ref{pos}, each $\alpha_p$ is nonzero, ensuring that $Z(0)$ is well-defined in Eq.~\eqref{Z}. Let $\Phi(x)$ be the Laplacian characteristic polynomial of $K_{n_1,\ldots,n_s}/F$. Since every Laplacian matrix is singular, we clearly have $\Phi(0) = 0$. On the other hand, Proposition \ref{lcp} implies
	\begin{equation*}
		\Phi(0) = \left( \prod_{p=1}^{c} (-\alpha_p) \right) \det Z(0) = 0.
	\end{equation*}
	As $\prod_{p=1}^{c} (-\alpha_p) \neq 0$ by Lemma \ref{pos}, we must have $\det Z(0) = 0$, as desired.
\end{proof}
\begin{proposition}\label{intra}
	\begin{equation}\label{tautr}
		\tau_F(K_{n_1,\ldots,n_s})=\frac{1}{c}\left(\prod_{p=1}^c\alpha_p\right)\tr\left(-\adj Z(0) \left.\frac{\mathrm{d}Z(x)}{\mathrm{d}x} \right|_{x=0}\right).
	\end{equation}
\end{proposition}
\begin{proof}
	Let $\Phi(x)=\Phi(K_{n_1,\ldots,n_s}/F;x)$. By Proposition \ref{lcp}, we have 
	\begin{equation*}
		\Phi(x)=\left(\prod_{p=1}^{c}(x-\alpha_p)\right)\det Z(x).
	\end{equation*}
	It follows from Corollary \ref{taud} that
	\begin{equation}\label{dprod}
		\tau_F(K_{n_1,\ldots,n_s}) = \frac{(-1)^{c-1}}{c} \left. \frac{\mathrm{d}}{\mathrm{d}x} \left[ \left( \prod_{p=1}^{c} (x-\alpha_p) \right) \det Z(x) \right] \right|_{x=0}.
	\end{equation} 
	By the  Leibniz product rule and the fact that $\det Z(0)=0$, Eq.~\eqref{dprod} can be simplified  as
	\begin{equation*}
		\tau_F(K_{n_1,\ldots,n_s})=\frac{(-1)^{c-1}}{c}\left(\prod_{p=1}^{c}(-\alpha_p)\right)\left.\frac{\mathrm{d}}{\mathrm{d}x}\det Z(x)\right|_{x=0}.
	\end{equation*}
Using Jacobi's formula (Lemma \ref{jac}), Eq.~\eqref{tautr} follows. The proof is complete.
\end{proof}
\section{Main result}
We first establish some basic properties of the $s \times s$ matrix $Z(0)$, which will facilitate a simpler but equivalent form for the term
\begin{equation*}
	\tr\left( -\adj Z(0) \left. \frac{\mathrm{d}Z(x)}{\mathrm{d}x} \right|_{x=0} \right)
\end{equation*}
appearing in the expression for $\tau_F(K_{n_1,\ldots,n_s})$.

\begin{lemma}\label{rz}
	The rank of $Z(0)$ is $s-1$.
\end{lemma}

\begin{proof}
	From Proposition \ref{Zsin}, we know that $Z(0)$ is singular, so $\rank Z(0) \leq s-1$. Suppose, toward a contradiction, that $\rank Z(0) < s-1$. Then $\adj Z(0)$ is the zero matrix, which implies that the trace term in Eq.~\eqref{tautr} vanishes. Consequently, by Proposition \ref{intra}, we would have $\tau_F(K_{n_1,\ldots,n_s}) = 0$. However, since the graph $K_{n_1,\ldots,n_s}/F$ is connected, its number of spanning trees must be positive, i.e., $\tau_F > 0$. This leads to a contradiction. Thus, $\rank Z(0) = s-1$.
\end{proof}

\begin{lemma}\label{lrk}
	Let $\mathbf{b} = (n_1,\ldots,n_s)^{\top}$. Then $\mathbf{b}^{\top} Z(0) = 0$ and $Z(0)(n\mathbf{e} - \mathbf{b}) = 0$.
\end{lemma}

\begin{proof}
	Let $b_j^{(1)}$ be the $j$-th entry of $\mathbf{b}^{\top} Z(0)$ for $j=1,\ldots,s$. By Eq.~\eqref{Z}, the $(i,j)$ entry of $Z(0)$ is 
	\begin{equation*}
		z_{ij}(0) = \delta_{ij} - \sum_{p=1}^c \frac{n_{jp}(m_p - n_{ip})}{\alpha_p}.
	\end{equation*}	
	Thus, for any $j \in \{1,\ldots,s\}$, we have
	\begin{align}\label{ai}
		b_j^{(1)} &= \sum_{i=1}^s n_i z_{ij}(0)\nonumber\\
		 &= n_j - \sum_{i=1}^s \sum_{p=1}^{c} \frac{n_i n_{jp}(m_p - n_{ip})}{\alpha_p} \nonumber\\
		&= n_j - \sum_{p=1}^{c} \frac{n_{jp}}{\alpha_p} \sum_{i=1}^s n_i(m_p - n_{ip}).
	\end{align}
	By Lemma \ref{pos}, $\alpha_p = \sum_{i=1}^s n_i(m_p - n_{ip})$. Hence, Eq.~\eqref{ai} reduces to 
	\begin{equation*}
		n_j - \sum_{p=1}^c n_{jp} = 0,
	\end{equation*}
	which follows from the second identity in Eq.~\eqref{rfp}. Next, we verify the second equality $Z(0)(n\mathbf{e} - \mathbf{b}) = 0$. Let $b_i^{(2)}$ be the $i$-th entry of $Z(0)(n\mathbf{e} - \mathbf{b})$. Noting that $\alpha_p = \sum_{j=1}^{s}(n - n_j)n_{jp}$ from Lemma \ref{pos}, a similar calculation yields
	\begin{align}
		b_i^{(2)} &= \sum_{j=1}^s z_{ij}(0)(n - n_j) \nonumber\\
		&= (n - n_i) - \sum_{j=1}^{s} \sum_{p=1}^c \frac{(n - n_j)n_{jp}(m_p - n_{ip})}{\alpha_p} \nonumber\\
		&= (n - n_i) - \sum_{p=1}^c \frac{(m_p - n_{ip})}{\alpha_p} \sum_{j=1}^{s} (n - n_j)n_{jp} \nonumber\\
		&= (n - n_i) - \sum_{p=1}^c (m_p - n_{ip}) \nonumber\\
		&= \left( n - \sum_{p=1}^c m_p \right) - \left( n_i - \sum_{p=1}^c n_{ip} \right)\nonumber\\
		& = 0. \nonumber
	\end{align}
	This completes the proof.
\end{proof}
A direct application of Lemma \ref{lrk} is the following characterization of the adjugate matrix of $Z(0)$.
\begin{proposition}\label{stradj}
Let $\mathbf{b}=(n_1,\ldots,n_s)^{\top}$. Then	$\adj Z(0)=\gamma(n\mathbf{e}-\mathbf{b})\mathbf{b}^{\top}$ for some constant $\gamma$. Moreover, for any $i,j\in \{1,\ldots,s\}$,
\begin{equation}\label{defC}
\gamma=\frac{C_{ij}}{n_i(n-n_j)},
\end{equation}
where $C_{ij}$ is the cofactor of $Z(0)$ corresponding to the $(i,j)$ entry.
\end{proposition}
\begin{proof}
	By Lemma \ref{rz}, we know that $\rank Z(0)=s-1$. Since $$Z(0)\adj Z(0)=(\det Z(0))I=0,$$
we find that each column of $\adj Z(0)$ belongs to the kernel of $Z(0)$. But this kernel is a one-dimensional space, spanned by the nonzero vector $(n\mathbf{e}-\mathbf{b})$ due to Lemma \ref{lrk}. Thus each column of $\adj Z(0)$ is a multiple of  $n\mathbf{e}-\mathbf{b}$. In other words, the matrix  $\adj Z(0)$ can be written in the form 
	$(n\mathbf{e}-\mathbf{b})\mathbf{f}^\top$ for some column vector $\mathbf{f}\in \mathbb{R}^s$. Similarly, as $(\adj Z(0))Z(0)=0$, we find $\mathbf{f}^\top Z(0)=0$ and hence $\mathbf{f}=\gamma \mathbf{b}$ for some constant $\gamma$. Therefore,
		$\adj Z(0)=\gamma(n\mathbf{e}-\mathbf{b})\mathbf{b}^\top$. By considering the $(j,i)$ entries of the two sides, we have
		\begin{equation*}
		C_{ij}=\gamma(n-n_j)n_i,\text{~i.e.,~} \gamma=\frac{C_{ij}}{n_i(n-n_j)}.
		\end{equation*}
		This completes the proof of Proposition \ref{stradj}.	
\end{proof}
Now we can simplify the trace expression stated at the beginning of this section.
\begin{proposition}\label{simtr}
	We have $$\tr\left(-\adj Z(0) \left.\frac{\mathrm{d}Z(x)}{\mathrm{d}x} \right|_{x=0}\right)=c\gamma,$$ where $\gamma$ is defined in Eq.~\eqref{defC}.	
\end{proposition}
\begin{proof}
 Let $d_{ij}$ be the $(i,j)$ entry of $\left.\frac{\mathrm{d}Z(x)}{\mathrm{d}x} \right|_{x=0}$ for $i,j\in\{1,\ldots,s\}$. By Definition \ref{defZ}, we have
 \begin{equation*}
 	d_{ij}=\left.\frac{\mathrm{d}z_{ij}}{\mathrm{d}x} \right|_{x=0}=-\sum_{p=1}^{c}\frac{n_{jp}(m_p-n_{ip})}{\alpha_p^2}.
 \end{equation*}
 Since $\tr(AB)=\tr(BA)$, we find by Proposition \ref{stradj} that, 
 \begin{align}
 \tr\left(-\adj Z(0) \left.\frac{\mathrm{d}Z(x)}{\mathrm{d}x} \right|_{x=0}\right)&=\tr\left(-\gamma(n\mathbf{e}-\mathbf{b})\cdot \mathbf{b}^\top\left.\frac{\mathrm{d}Z(x)}{\mathrm{d}x} \right|_{x=0} \right)\nonumber\\
 &=\gamma\tr\left(- \mathbf{b}^\top \left.\frac{\mathrm{d}Z(x)}{\mathrm{d}x} \right|_{x=0}\cdot(n\mathbf{e}-\mathbf{b}) \right)\nonumber\\
 &=\gamma \sum_{i=1}^s\sum_{j=1}^s n_i(- d_{ij})(n-n_j)\nonumber\\
 &=\gamma \sum_{i=1}^s\sum_{j=1}^s\sum_{p=1}^c n_i\frac{n_{jp}(m_p-n_{ip})}{\alpha_p^2}(n-n_j)\nonumber\\
 &=\gamma\sum_{p=1}^c\frac{1}{\alpha_p^2} \sum_{i=1}^s n_i(m_p-n_{ip})\sum_{j=1}^s (n-n_j) n_{jp}\nonumber\\
 &=\gamma\sum_{p=1}^c\frac{1}{\alpha_p^2}\times \alpha_p\times \alpha_p\nonumber\\
 &=c\gamma\nonumber,
 \end{align}
 where the second to last equality follows from the first statement of Lemma \ref{pos}.
\end{proof}
Now we are in a position to present the main result of this paper.
\begin{theorem}\label{main}
	For a spanning forest of $K_{n_1,\ldots,n_s}$ with $c$ components,
	\begin{equation}\label{gf}
	\tau_F(K_{n_1,\ldots,n_s})=\frac{C_{ij}}{n_i(n-n_j)}\prod_{p=1}^c \alpha_p \text{~for any~} i,j\in \{1,\ldots,s\},
	\end{equation}
	where $C_{ij}$ is the cofactor of the  matrix 
	\begin{equation}\label{mZ}
		I_s-\left(\sum_{p=1}^c\frac{n_{jp}(m_p-n_{ip})}{\alpha_p}\right)_{s\times s}
	\end{equation}
	corresponding to the $(i,j)$ entry.
\end{theorem}
\begin{proof}
	By Propositions \ref{intra} and \ref{simtr}, Theorem \ref{main} follows.
\end{proof}
For the special case $s=2$, it is straightforward to recover the result of Dong and Ge \cite{dong2022}. Letting  $i=j=1$, we have, from Eq.~\eqref{gf}, 
\begin{equation}\label{dg0}	\tau_F(K_{n_1,n_2})=\frac{C_{11}}{n_1n_2}\prod_{p=1}^c\alpha_p.
	\end{equation}
By Lemma \ref{pos}, $\alpha_p=n_{1p}n_2+n_{2p}n_1$. Note that $C_{11}$ is the right corner of the $2\times s$ matrix described in Eq.~\eqref{mZ}. We have  
\begin{equation*}
	C_{11}=1-\sum_{p=1}^c\frac{n_{2p}(m_p-n_{2p})}{\alpha_p}=1-\sum_{p=1}^c\frac{n_{1p}n_{2p}}{n_{1p}n_2+n_{2p}n_1}.
\end{equation*}
It follows  from Eq.~\eqref{dg0} that 
\begin{equation*}	\tau_F(K_{n_1,n_2})=\frac{1}{n_1n_2}\left(\prod_{p=1}^c (n_{1p}n_2+n_{2p}n_1)\right)\left(1-\sum_{p=1}^c\frac{n_{1p}n_{2p}}{n_{1p}n_2+n_{2p}n_1}\right),
\end{equation*}
which is exactly the formula obtained by Dong and Ge \cite{dong2022}.	
	\section*{Declaration of competing interest}
	There is no conflict of interest.

\section*{Acknowledgements}
This work is partially supported by the National Natural Science Foundation of China (Grant Nos. 12371355 and 12001006) and  Wuhu Science and Technology Project, China (Grant No.~2024kj015).

\end{document}